\documentclass[12pt]{amsart}

\makeatletter
\g@addto@macro{\endabstract}{\@setabstract}
\makeatother

\usepackage{graphics, stackrel}
\usepackage{amsmath, amssymb, amsthm}
\usepackage{graphicx}
\usepackage{verbatim}
\usepackage{amsfonts}
\usepackage[numbers]{natbib}
\usepackage{enumitem}

\usepackage{color}
\usepackage{mdwlist}

\usepackage{hyperref}
\hypersetup{citecolor=blue, colorlinks=true, linkcolor=blue}

\usepackage{mathrsfs}
\usepackage{bbm}

\usepackage[left=1.25in, right=1.25in, top=1.0in, bottom=1.15in, includehead, includefoot]{geometry}

\DeclareMathOperator{\osc}{osc}

\newcommand{\setntn}[2]{ \{ #1 : #2 \} }

\newcommand{\preqsd}{\preceq_{sd} }

\newcommand{\1}{\mathbbm 1}

\newcommand{\dD}{\mathcal D}

\newcommand{\bB}{\mathcal B}

\newcommand{\hH}{\mathcal H}

\newcommand{\mM}{\mathcal M}

\newcommand{\fF}{\mathcal F}

\newcommand{\pP}{\mathcal P}

\newcommand{\RR}{\mathbbm R}

\newcommand{\PP}{\mathbbm P}

\newcommand{\EE}{\mathbbm E \,}

\theoremstyle{plain}
\newtheorem{theorem}{Theorem}[section]

\newtheorem{proposition}{Proposition}[section]

\theoremstyle{definition}

\newtheorem{axiom}{Axiom}[section]

\begin{document}

\title{}

\date{\today}

\begin{center}
  \LARGE 
    Partial Stochastic Dominance via \\ Optimal
    Transport\footnote{Corresponding author:
        \texttt{john.stachurski@anu.edu.au},
    RSE, College of Business and Economics, Australian National University,
    ACT, Australia.} 
    \vspace{1em}

  \large
  Takashi Kamihigashi\textsuperscript{a} and John Stachurski\textsuperscript{b} \par \bigskip

  \small
  \textsuperscript{a}Center for Computational Social Science, Kobe University \par
  \textsuperscript{b}Research School of Economics, Australian National University \bigskip

  \normalsize
  \today
\end{center}

\begin{abstract}
    In recent years, a range of measures of ``partial'' stochastic dominance have been introduced.  These measures attempt to determine the extent to which one distribution is dominated by another.  We assess these measures from intuitive, axiomatic, computational and statistical perspectives.  Our investigation leads us to recommend a measure related to optimal transport as a natural default.
    \vspace{1em}

\end{abstract}


\section{Introduction}

(First order) stochastic dominance is one of the most fundamental concepts in
social welfare and decision making under uncertainty (see, e.g.,
\cite{follmer2011stochastic, stoyanov2012metrization}).  At the same time, stochastic dominance is
fragile.  For example, two normal distributions can only be ordered by
stochastic dominance if their
variances are exactly identical, even if one mean is orders of magnitude
larger than the other.  Such fragility is problematic for quantitative
work.



In response, researchers have introduced many different notions of ``partial''
stochastic dominance.  One is ``restricted stochastic dominance,'' which
compares the order of cumulative distribution functions ({\sc cdf}s) only up to some
specified point in the domain \cite{atkinson1987measurement,
davidson2000statistical, davidson2013testing}.  Another is ``almost
stochastic dominance,'' which was developed in the finance literature
\cite{leshno2002preferred, levy2009almost}.    A third kind of
measure was proposed in \cite{fields2002stochastic}, analyzing degree of
stochastic dominance in the context of income mobility analysis.  Still more
measures are considered in \cite{stoyanov2012metrization}.

Despite the obvious practical relevance of a measure of partial stochastic
dominance, none of the above have a clear axiomatic foundation.  At the same
time, some measures become complex outside of the one-dimensional case.  These
observations suggest that now is a good time to consider measures of partial
stochastic dominance collectively.  While doing so, we make the case for what
we believe is the most natural measure of partial stochastic dominance.  The
measure can be understood as the solution to an optimal transport problem.

\section{A Measure of Dominance}\label{s:am}

Let $S$ be Polish with Borel sets $\bB$ and closed partial order $\preceq$
(i.e., a partial order on $S$ such that $\setntn{(x,y) \in S \times S}{x
\preceq y}$ is closed in the product topology).  A function $h \colon S \to
\RR$ is called \emph{increasing} if $x \preceq y$ implies $h(x) \leq h(y)$,
and \emph{decreasing} if $-h$ is increasing.  Let $\hH$ be the set of increasing
Borel measurable $h$ on $S$ with $\osc(h) \leq 1$ (that is, $\sup h - \inf h
\leq 1$).  Let $\mM$ be all finite measures on $(S, \bB)$ and $\pP$ be all
$\mu \in \mM$ with $\mu(S)=1$.
A measure $\mu \in \mM$ is said to be \emph{stochastically dominated} by $\nu
\in \mM$ and we write $\mu \preqsd \nu$ if $\mu(S) = \nu(S)$ and $\int h \, d
\mu \leq \int h \, d\nu$ for all $h \in \hH$.
For probability measures $\mu, \nu \in \pP$, 
an alternative characterization is
\begin{equation}\label{eq:esd}
    \mu \preqsd \nu \;\; \;
    \text{ if and only if }
    \; \;
    \max_{(X, Y) \in \Pi(\mu, \nu)} \PP\{X \preceq Y \} = 1,
\end{equation}
where $\Pi(\mu, \nu)$ is the set of all couplings of $\mu$ and
$\nu$ \cite{strassen1965existence}.  (Given $(\mu, \nu) \in \pP \times \pP$, a pair of $S$-valued random
variables $(X,Y)$ defined on some probability space $(\Omega, \fF, \PP)$ is
called a \emph{coupling} of $(\mu, \nu)$ if $\mu(B) = \PP\{X \in B\}$ and
$\nu(B) = \PP\{Y \in B\}$ for all $B \in \bB$.) The short summary of our paper
is: in light of~\eqref{eq:esd}, why not use 
\begin{equation}\label{eq:deftau}
    \tau(\mu, \nu) := \max_{(X, Y) \in \Pi(\mu, \nu)} \PP\{X \preceq Y \} 
\end{equation}
as the default measure of partial stochastic dominance?
We answer this question in stages.  Existence of the maximum for each $(\mu,
\nu)$ pair is verified below.

Note that we can rewrite $\tau$ as
$1 - \min_{(X, Y) \in \Pi(\mu, \nu)} \EE \, c(X, Y)$,
where $c(x, y) = \1\{x \npreceq y\}$.  The infimum is a standard 
optimal transport problem.  In particular, $c$ is nonnegative, bounded and
lower-semicontinuous (since $\preceq$ is a closed partial order), so, by 
Theorem~4.1 of \cite{villani2008optimal}, a solution exists.
In particular, the ``max'' in~\eqref{eq:deftau} is justified.
From now on, we call a pair $(X, Y)$ that attains the maximum
in~\eqref{eq:deftau} an \emph{optimal coupling}.

It now follows from Kantorovich duality that 
$\min_{(X, Y) \in \Pi(\mu, \nu)} \EE \, c(X, Y)$
can be replaced by a supremum over supporting ``prices.'' 
In particular,
\begin{equation}\label{eq:deftau2}
    \tau(\mu, \nu) 
    = 1 - \max_{g \in \dD} \left\{ \int g d\nu - \int g d\mu \right\}.
\end{equation}
where $\dD$ is all $g \colon S \to \RR$ with $g(y) - g(x) \leq \1\{x \npreceq
y\}$.  This is (iii) of Theorem~5.10 of~\cite{villani2008optimal}.  (See
``Particular case 5.16'' on p.~60.)  The
existence of a maximizer is guaranteed by the same theorem.  Clearly
$\dD$ is just the set of decreasing functions on $S$ with $\osc(h) \leq 1$.  
Hence we can rewrite \eqref{eq:deftau2} as
\begin{equation}\label{eq:deftau3}
    \tau(\mu, \nu) 
    = 1 - \max_{h \in \hH} \left\{ \int h d\mu - \int h d\nu \right\}.
\end{equation}
These expressions have additional representations that can be useful in some
settings.  For example, if $\mu = F$ and $\nu = G$ are {\sc cdf}s, then one can
show
$\tau$ has the simple representation
\begin{equation}
    \label{eq:im}
    \tau(F, G) = 1 - \sup_x \{ G(x) - F(x) \}.
\end{equation}

The representation in \eqref{eq:deftau3} is very similar to measures of
partial stochastic dominance studied in \cite{stoyanov2012metrization}.  They
use essentially the same measure but replace $\hH$ with the set of
all increasing functions satisfying a Lipschitz bound with respect to the
underlying metric.  Such a measure is more
attuned to topology but lacks the some of the advantages of~\eqref{eq:deftau3}
described below.

\section{Alternative Measures}

\label{ss:em}

Here and below, a \emph{measure of degree of stochastic dominance} is any
function 
\begin{equation}
    \label{eq:dfdm}
    \delta \colon \pP \times \pP \to [0, 1] \;
    \text{ such that }
    \mu \preqsd \nu   
    \; \implies \;
    \delta(\mu, \nu) = 1.
\end{equation}
Clearly $\tau$ is a measure of degree of stochastic dominance
in the sense of \eqref{eq:dfdm}.
Another is the quantile ratio
measure proposed in \cite{fields2002stochastic}.  Let $S = [a, b]$ and let
$\preceq$ be the usual order $\leq$ on $\RR$.  Letting $F$ and $G$ be {\sc
cdf}s on
$[a, b]$, define
\begin{equation}
    \label{eq:mq}
    q(F, G) := \frac{m}{n}
    \quad \text{where} \quad
   m := \sum_{i=1}^n \1\{G(x_i) \leq F(x_i)\}.
\end{equation}
Here $\{x_i\}$ is a grid of $n$ specified values in $[a, b]$, typically
corresponding to some quantile points.  Thus $q(F, G)$ measures the fraction
of times that $G(x_i) \leq F(x_i)$ is observed over specified test points.
Clearly $q$ satisfies $q(F, G) = 1$ when $G \leq F$ pointwise (i.e., $F
\preqsd G$), which corresponds to \eqref{eq:dfdm}.

Another version of partial stochastic dominance is
\emph{restricted} stochastic dominance (see, e.g.
\cite{atkinson1987measurement, davidson2013testing}).  Let $F$ and
$G$ be one dimensional {\sc cdf}s on an interval $[a, b]$.  Given $c \in [a, b]$,
distribution $F$ is said to be dominated by $G$ in the restricted sense if
$G(x) \leq F(x)$ for all $x \leq c$.  We can turn this into a measure by
considering the largest such $c$, defined by
\begin{equation}
    \label{eq:cstar}
    c^*(F, G) := \sup \{c \in [a, b] \,:\, G(x) \leq F(x), \; \forall \, x \leq c\}.
\end{equation}
After normalizing we get 
\begin{equation}
    \label{eq:mr}
    r(F, G) := \frac{c^*(F, G) - a}{b - a} .
\end{equation}
Evidently $r$ is a measure of degree of stochastic dominance in the sense
of \eqref{eq:dfdm}.


Another measurement for partial stochastic dominance is 
\emph{almost} stochastic dominance \cite{levy1992stochastic,
    leshno2002preferred}.  Once again the context is $S = [a,b]$ with the
    usual order $\leq$.  For {\sc cdf}s $F$ and $G$ on $[a, b]$, the measure can be
expressed as
\begin{equation}
    \label{eq:ma}
    \alpha(F, G)
        := \frac{\int (F(x) - G(x))_+ \; dx}{\int |F(x) - G(x)| \; dx}
\end{equation}
where $v_+ := \max\{v, 0\}$ for any $v \in \RR$. 
Intuitively, if $F$ is almost dominated by $G$, then $G \leq F$ on most of its
domain, and $(F(x) - G(x))_+ =
|F(x) - G(x)|$ for most $x$. Hence $\alpha(F, G)$ is close to 1.  In order to
ensure that the measure is defined for all pairs $F, G$, 
we adopt the convention that $\alpha(F, G) = 1$ when $F=G$.


\section{Axioms}

\label{s:a}

Next we propose two axioms for degree of
stochastic dominance $\delta(\mu, \nu)$.  
The first says that $\delta(\mu, \nu)$ should not be large
unless the distributions are nearly ordered.  To state it, we 
write $\mu \leq \nu$ to indicate pointwise ordering on $\bB$
and define, for each $\mu, \nu \in \pP$, 
\begin{equation*}
    \Phi(\mu, \nu) 
    := \setntn{ (\mu', \nu') \in \mM \times \mM }
        {\mu' \leq \mu,\; \nu' \leq \nu,\; \mu' \preqsd \nu'}.
\end{equation*}
Think of $\Phi(\mu, \nu)$ as the set of ``ordered component pairs''
corresponding to $(\mu, \nu)$.  If $\mu$ is ``almost''
dominated by $\nu$, then we can choose relatively large components.
(Recall that, to admit the ordering
$\mu' \preqsd \nu'$, we insist that total
mass is equal, so $\mu'(S) = \nu'(S)$ must hold.)

\begin{axiom}\label{a:c}
    For each $(\mu, \nu) \in \pP \times \pP$ and $\epsilon > 0$,
    there exists a $(\mu', \nu')$ in $\Phi(\mu, \nu)$ such that
    $ \delta(\mu, \nu) \leq \mu'(S) + \epsilon$.
\end{axiom}

Figure~\ref{f:dms} helps to illustrate the axiom, with $\mu$ and $\nu$
represented by densities.  In the top subfigure, $\mu$ is in no sense
dominated by $\nu$, so we wish to enforce $\delta(\mu, \nu) = 0$.
Axiom~\ref{a:c} does enforce this, since, for this pair $(\mu, \nu)$, we
cannot extract an ordered component pair $(\mu', \nu')$ with positive mass.
Hence $\mu'(S) = \nu'(S)$ is always zero. 

In the lower subfigure, there is some overlap in probability mass, so we
should permit $\delta(\mu, \nu) > 0$.  Inspection shows this to be true.
For example, we could take both $\mu'$ and $\nu'$ to be the function enclosing the
shaded region (the pointwise infimum $\mu \wedge \nu$ of the density representations of $\mu$ and $\nu$), so that, as measures again, $\mu'(S) = \nu'(S) =$ the area of the shaded region.
Since this is positive, $\delta(\mu, \nu)$ can be positive.

\begin{figure}
    \centering
    \scalebox{0.65}{\includegraphics[clip=true, trim=0mm 5mm 0mm 5mm]{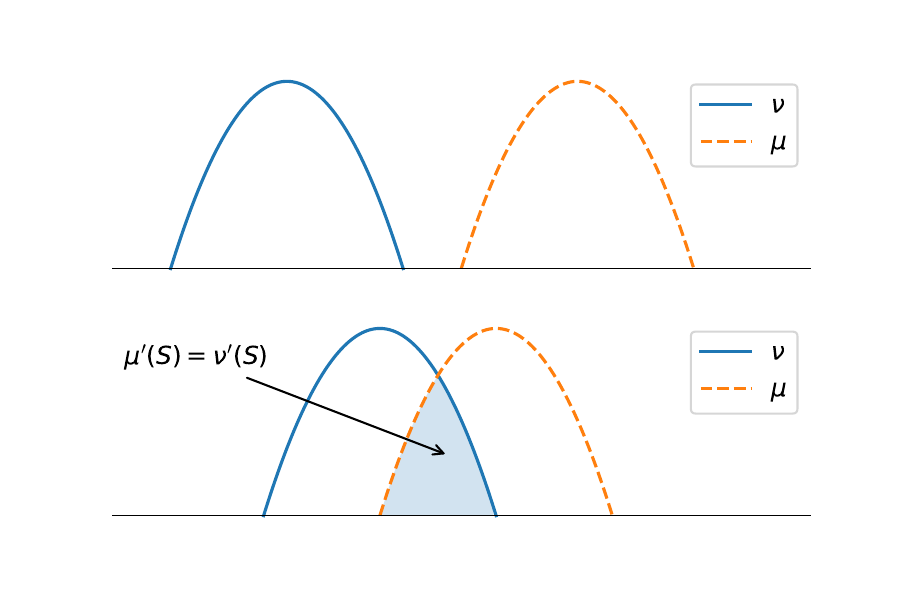}}
    \caption{\label{f:dms} Densities and ordered component pairs}
\end{figure}

Axiom~\ref{a:c} adds the converse implication to \eqref{eq:dfdm}, so
$\delta$ identifies ordered pairs: 

\begin{proposition}\label{p:iff}
    If $\delta$ is a measure of degree of stochastic dominance that satisfies
    Axiom~\ref{a:c}, then $\delta(\mu, \nu) = 1$ if and only if $\mu \preqsd \nu$.
\end{proposition}

\begin{proof}
    Suppose that $\delta$ satisfies Axiom~\ref{a:c} and that $\delta(\mu,
    \nu) = 1$.  Then $\mu'(S) = \nu'(S) = 1$ for some $(\mu', \nu')$ in
    $\Phi(\mu, \nu)$. From this we immediately have $\mu' = \mu$ and $\nu' =
    \nu$, and hence $\mu \preqsd \nu$, which is all we need to show.
\end{proof}


The second axiom plays the opposite role.  It implies that $\delta(\mu, \nu)$
is close to $1$ when $\mu$ is ``nearly'' dominated by $\nu$.  

\begin{axiom}\label{a:mix}
    Given $\mu_a, \mu_b, \nu_a, \nu_b \in \pP$ and $\lambda \in [0, 1]$, 
    the ordering $\mu_a \preqsd \nu_a$ implies $\delta(\mu, \nu) \geq \lambda$
    for $\mu := \lambda \mu_a + (1 - \lambda) \mu_b$
        and $\nu := \lambda \nu_a + (1 - \lambda) \nu_b$.
\end{axiom}

Axiom~\ref{a:mix} uses the natural convexity of $\pP$ to implement
``continuity near 1'' while avoiding being tied to a particular topology.

The axioms we have listed are strong enough to give uniqueness:  

\begin{theorem}\label{t:uniqueness}
    The only measure of degree of stochastic dominance
    satisfying Axioms~\ref{a:c}--\ref{a:mix} is $\tau$.
\end{theorem}

\begin{proof}
    First we claim that $\tau$ satisfies Axiom~\ref{a:c}.
    To see this, let $(X, Y)$ be an optimal coupling for $(\mu, \nu)$,
    attaining the maximum in \eqref{eq:deftau}.  Set 
        $\mu'(B) = \PP\{X \in B, \, X \preceq Y\}$
        and $\nu'(B) = \PP\{Y \in B, \, X \preceq Y\}$.
    This pair satisfies $(\mu', \nu') \in \Phi(\mu, \nu)$ and $\mu'(S) =
    \tau(\mu, \nu)$.
    To see that $\tau$ satisfies Axiom~\ref{a:mix},
    fix $\mu, \nu \in \pP$ and let the decompositions in Axiom~\ref{a:mix} be
    given, with $\lambda \in [0, 1]$ and $\mu_a \preqsd \nu_a$. 
    Let $(X_i, Y_i)$ be an optimal coupling of $(\mu_i, \nu_i)$
    for $i=a,b$.  Let $X = \ell X_a + (1 - \ell)X_b$
    where $\ell$ is independent and binary with $\PP\{\ell=1\} = \lambda$.
    Let $Y = \ell Y_a + (1 - \ell) Y_b$.  Then $(X, Y) \in \Pi(\mu, \nu)$.
    Hence $\tau(\mu, \nu) \geq \PP\{X \preceq Y\} \geq \lambda \PP\{X_a
    \preceq Y_a\} = \lambda$, and Axiom~\ref{a:mix} holds.

    Now let $\delta$ be an arbitrary measure of degree of stochastic dominance
    satisfying Axioms~\ref{a:c}--\ref{a:mix}.  Fix $(\mu, \nu)$ in $\pP \times
    \pP$ and let $(X, Y)$ be an optimal coupling of $(\mu, \nu)$.
    Define $\lambda := \PP\{X \preceq Y\}$ and the probabilities
    \begin{align*}
        \mu_a(B) & = \frac{\PP\{ X \in B, \; X \preceq Y \} }{ \lambda },
        \quad
        \nu_a(B) = \frac{ \PP\{ Y \in B, \; X \preceq Y \} }{ \lambda }, 
        \\
        \mu_b(B) & = \frac{\PP\{ X \in B, \; X \npreceq Y \} }{ 1 - \lambda},
        \quad
        \nu_b(B) = \frac{ \PP\{ Y \in B, \; X \npreceq Y \} } {1 - \lambda}.
    \end{align*}
    We have $\mu_a \preqsd \nu_a$, since, when $I \in \bB$ is increasing,
    \begin{equation*}
        \mu_a(I) 
        = \frac{\PP\{ X \in I, \; X \preceq Y \}}{ \lambda }
        \leq \frac{ \PP\{ Y \in I, \; X \preceq Y \} } {\lambda }
        = \nu_a(I).
    \end{equation*}
    For $\mu$ we have
    $\mu(B) = \PP\{ X \in B \}$, which can be decomposed as
        $\PP\{ X \in B, \; X \preceq Y \} 
            + \PP\{ X \in B, \; X \npreceq Y \} 
            = \lambda \mu_a(B) + (1 - \lambda) \mu_b(B)$.
    Since $\delta$ satisfies Axiom~\ref{a:mix}, we have $\delta(\mu, \nu) \geq
    \lambda = \PP\{ X \preceq Y \} = \tau(\mu, \nu)$.

    For the reverse inequality, recall from the arguments above that we can
    obtain an ordered component pair $(\mu', \nu')$ satisfying $\nu'(S) =
    \mu'(S) = \tau(\mu, \nu)$.  Since $\delta$ satisfies Axiom~\ref{a:c}, we
    have $\delta(\mu, \nu) \leq \tau(\mu, \nu) + \epsilon$ for all $\epsilon >
    0$.  Hence  $\delta(\mu, \nu) \leq \tau(\mu, \nu)$.
\end{proof}

\section{Measures vs Axioms}

\label{s:meax}

Let us reconsider measures of stochastic dominance other than $\tau$.  
By Theorem~\ref{t:uniqueness}, they fail at least one of the axioms.
For example, regarding Axiom~\ref{a:mix}, note that $q$ fails 
whenever the grid $\{x_i\}$ has at least three points.  To see this let $F'$,
$F''$ and $G''$ be any {\sc cdf}s on $[a, b]$ such that $F'' < G''$ on $(a, b)$.
Let $\lambda \in (0, 1)$ and let
\begin{equation*}
    F = \lambda F' + (1 - \lambda) F'', 
    \quad
    G = \lambda F' + (1 - \lambda) G''. 
\end{equation*}
Since $F' \preqsd F'$, Axiom~\ref{a:mix} implies that $q(F, G) \geq \lambda$.
On the other hand, $G(x_i) > F(x_i)$ on any interior point $x_i$.  Hence $m$
in \eqref{eq:mq} is at most 2, and $q(F, G) \leq 2/n$. Since
$\lambda$ can be arbitrarily close to 1 this contradicts Axiom~\ref{a:mix}.

For this same pair $F, G$, the fact that $G > F$ on $(a, b)$ implies that
$c^*(F, G) = 0$ for $c^*$ defined in \eqref{eq:cstar}, and hence $r(F, G) = 0$
for the restricted stochastic dominance measure defined in \eqref{eq:mr}.
Likewise, for the same pair, $\alpha(F, G) = 0$, where $\alpha$ is the almost
stochastic dominance measure.  Hence $r$ and $\alpha$ also fail to satisfy
Axiom~\ref{a:mix}.

Regarding Axiom~\ref{a:c}, the measures $q$, $r$ and $\alpha$ all fail.
To see this, let $S = [0, 1]$ and, given some positive number $\epsilon$,
let $\mu$ put mass $\epsilon$ on $0$ and $1 - \epsilon$ on $1$, and let $\nu$
put all mass on $1-\epsilon$.   Let $F$ be the {\sc cdf} of $\mu$ and let $X$ be a
draw from $\mu$.  Let $G$ and $Y$ be the {\sc cdf} of and a draw from $\nu$
respectively.  Since $Y$ is certainly $1-\epsilon$ we have $X \leq Y$ if and
only if $X = 0$, and hence $\PP\{X \leq Y\} = \epsilon$ for all couplings.  On
the other hand, $F(x) > G(x)$ iff $0 \leq x < 1-\epsilon$, and hence $r(F, G)
= 1 - \epsilon$.  If $\epsilon < 1/2$ then $1 - \epsilon > \epsilon$, and
hence $r$ fails Axiom~\ref{a:c}.
The measures $q$ and $\alpha$ also give values larger than $\epsilon$ when
$\epsilon$ is small, although the details are
omitted.

\section{Final Comments}

\label{s:c}

We end with some avenues for future research.  One is that an
estimation theory for $\tau$ should be straightforward to construct.  For
example, if we replace the {\sc cdf}s $F$ and $G$ in~\eqref{eq:im} with
empirical counterparts $F_n$ and $G_n$, then $\tau(F_n, G_n)$ is a simple
transform of the statistic used in the one-sided two-sample Kolmogorov-Smirnov
test.  This allows for construction of confidence intervals and
hypotheses tests related to the value of $\tau$.

Another topic of interest is computation.  In higher dimensional settings,
computation is difficult for all measures of partial stochastic dominance, but
$\tau$ at least has an interpretation as the solution to an optimal transport
problem.  Computing solutions of optimal transport problems is an active
research area \cite{peyre2019computational}.  

Third, there is some connection between the measure $\tau$, which takes values
in the interval $[0, 1]$ and represents the continuum between no dominance and
complete first order dominance, and the measure proposed in
\cite{muller2017between}, which represents the continuum between first order
and second order stochastic dominance. Clarifying this connection and
investigating the
preceding two topics are left for future work.

\section{Acknowledgements}

We gratefully acknowledge JSPS KAKENHI Grant 15H05729 and Australian Research
Council Grant DP120100321.

\bibliographystyle{plain}

\bibliography{sd_bib}

\begin{thebibliography}{10}

\bibitem{atkinson1987measurement}
Anthony~B Atkinson.
\newblock On the measurement of poverty.
\newblock {\em Econometrica}, 55(4):749--764, 1987.

\bibitem{davidson2000statistical}
Russell Davidson and Jean-Yves Duclos.
\newblock Statistical inference for stochastic dominance and for the measurement of poverty and inequality.
\newblock {\em Econometrica}, 68(6):1435--1464, 2000.

\bibitem{davidson2013testing}
Russell Davidson and Jean-Yves Duclos.
\newblock Testing for restricted stochastic dominance.
\newblock {\em Econometric Reviews}, 32(1):84--125, 2013.

\bibitem{fields2002stochastic}
Gary~S Fields, Jesse~B Leary, and Efe~A Ok.
\newblock Stochastic dominance in mobility analysis.
\newblock {\em Economics Letters}, 75(3):333--339, 2002.

\bibitem{follmer2011stochastic}
Hans F{\"o}llmer and Alexander Schied.
\newblock {\em Stochastic Finance: An Introduction in Discrete Time}.
\newblock De Gruyter Textbook Series. De Gruyter, 2011.

\bibitem{leshno2002preferred}
Moshe Leshno and Haim Levy.
\newblock Preferred by “all” and preferred by “most” decision makers: Almost stochastic dominance.
\newblock {\em Management Science}, 48(8):1074--1085, 2002.

\bibitem{levy1992stochastic}
Haim Levy.
\newblock Stochastic dominance and expected utility: survey and analysis.
\newblock {\em Management Science}, 38(4):555--593, 1992.

\bibitem{levy2009almost}
Moshe Levy.
\newblock Almost stochastic dominance and stocks for the long run.
\newblock {\em European Journal of Operational Research}, 194(1):250--257, 2009.

\bibitem{muller2017between}
Alfred M{\"u}ller, Marco Scarsini, Ilia Tsetlin, and Robert~L Winkler.
\newblock Between first-and second-order stochastic dominance.
\newblock {\em Management Science}, 63(9):2933--2947, 2017.

\bibitem{peyre2019computational}
Gabriel Peyr{\'e} and Marco Cuturi.
\newblock {\em Computational optimal transport}.
\newblock Now Publishers, 2019.

\bibitem{stoyanov2012metrization}
Stoyan~V Stoyanov, Svetlozar~T Rachev, and Frank~J Fabozzi.
\newblock Metrization of stochastic dominance rules.
\newblock {\em International Journal of Theoretical and Applied Finance}, 15(02), 2012.

\bibitem{strassen1965existence}
Volker Strassen.
\newblock The existence of probability measures with given marginals.
\newblock {\em The Annals of Mathematical Statistics}, pages 423--439, 1965.

\bibitem{villani2008optimal}
C{\'e}dric Villani.
\newblock {\em Optimal transport: old and new}.
\newblock Springer Science, 2008.

\end{thebibliography}

\end{document}